\documentclass[12pt]{amsart}
\usepackage[english]{babel}
\usepackage{graphicx}
\usepackage[T1]{fontenc}
\usepackage{color}

\newtheorem{Thm}{Theorem}

\newtheorem{Lem}{Lemma}
\newtheorem{Rem}{Remark}

\newtheorem*{Prop*}{Proposition}
\newtheorem*{Cor*}{Corollary}
\newtheorem*{Thm*}{Theorem}

\def\eps{\varepsilon}
\def\R{{\mathbb R}}

\def\S{{\mathbb S}}

\makeatletter
\newsavebox{\@brx}
\newcommand{\llangle}[1][]{\savebox{\@brx}{\(\m@th{#1\langle}\)}%
  \mathopen{\copy\@brx\mkern2mu\kern-0.9\wd\@brx\usebox{\@brx}}}
\newcommand{\rrangle}[1][]{\savebox{\@brx}{\(\m@th{#1\rangle}\)}%
  \mathclose{\copy\@brx\mkern2mu\kern-0.9\wd\@brx\usebox{\@brx}}}
\makeatother

\begin{document}

\title[Asymptotic Analysis of linear Boltzmann equation]{simultaneous diffusion and homogenization asymptotic\\ for the linear Boltzmann equation}

\bibliographystyle{plain}

\author[Claude Bardos]{Claude Bardos}
\address{C.B.: LJLL, University of Paris VII, Jussieu, Paris, France.} 
\email{claude.bardos@gmail.com}

\author[Harsha Hutridurga]{Harsha Hutridurga}
\address{H.H.: Department of Mathematics, Imperial College London, London, SW7 2AZ, United Kingdom.}
\email{h.hutridurga-ramaiah@imperial.ac.uk}

\begin{abstract}
This article is on the simultaneous diffusion approximation and homogenization of the linear Boltzmann equation when both the mean free path $\eps$ and the heterogeneity length scale $\eta$ vanish. No periodicity assumption is made on the scattering coefficient of the background material. There is an assumption made on the heterogeneity length scale $\eta$ that it scales as $\eps^\beta$ for $\beta\in(0,\infty)$. In one space dimension, we prove that the solutions to the kinetic model converge to the solutions of an effective diffusion equation for any $\beta\le2$ in the $\eps\to0$ limit. In any arbitrary phase space dimension, under a smallness assumption of a certain quotient involving the scattering coefficient in the $H^{-\frac{1}{2}}$ norm, we again prove that the solutions to the kinetic model converge to the solutions of an effective diffusion equation in the $\eps\to0$ limit.
\end{abstract}

\maketitle

\setcounter{tocdepth}{1}
\tableofcontents

\section{Introduction}\label{sec:intro}

The unknown quantity modeled by equations in kinetic theory is the probability distribution function $f(t,x,v)$ of a population of particles which is a function of time, position and velocity. The linear Boltzmann equation describes the evolution of the distribution function modeling the collision of a population of particles with a background medium:
\begin{align*}
\partial_t f(t,x,v)
+
v \cdot \nabla_x f(t,x,v)
=
\sigma(x)
\left(
\int_{\mathcal{V}}
f(t,x,w)
\, {\rm d}\mu(w)
-
f(t,x,v)
\right)
\end{align*}
where $\mu$ is a Borel probability measure on the space of velocities $\mathcal{V}$. The coefficient $\sigma(x)$ is the scattering coefficient of the background material. Our objective is to perform an asymptotic analysis when the background medium is inhomogeneous -- say, a composite material with microstructure. We consider two asymptotically small parameters associated with the above kinetic model: (a) mean free path $\eps$ of the particles between two interactions; (b) the scale of heterogeneity $\eta$ of the background medium -- it can be the average distance between two neighboring inhomogeneities or the average size of the inhomogeneities.

In the above kinetic model, the scattering coefficient $\sigma(x)$ represents the background medium. The inhomogeneous nature of the background medium implies that smaller the parameter $\eta$ is, more rapid the oscillations are in $\sigma(x)$. As is standard in the theory of homogenization, we consider a family of scattering coefficients indexed by $\eta$, i.e. $\sigma^\eta(x)$ and study an associated family of solutions to the kinetic model. We also wish to study the evolution of the local equilibria for the above kinetic model. This corresponds to scaling the above kinetic model using parabolic scaling with the parameter $\eps$. The objective of this article is to study the following scaled linear Boltzmann equation
\begin{align*}
\eps \partial_t f^{\eps,\eta}(t,x,v)
+
v\cdot \nabla_x f^{\eps,\eta}(t,x,v)
=
\frac{\sigma^\eta(x)}{\eps}
\left(
\int_{\mathcal{V}}
f^{\eps,\eta}(t,x,w)
\, {\rm d}\mu(w)
-
f^{\eps,\eta}(t,x,v)
\right)
\end{align*}
in the simultaneous limit as both the scaling parameters $\eps$, $\eta$ vanish. The diffusion approximation of the linear Boltzmann equation corresponds to the $\eps\to0$ limit in the above scaled equation (see \cite{bardos2013diffusion} and references therein for the state-of-the-art on the techniques used in the diffusion approximation of linear transport equations). Therefore in the regime $\eps\ll \eta$, one can first perform the diffusion asymptotic for a fixed $\eta$ yielding a parabolic equation with heterogeneous coefficient $\sigma^\eta(x)$. The $\eta\to0$ limit corresponds to deriving an homogenized equation for the thus obtained heterogeneous parabolic equation. This can either be performed using the asymptotic expansions method \cite{bensoussan1978asymptotic} or the two-scale convergence method \cite{nguetseng1989general, allaire1992homogenization} when $\sigma^\eta(x)=\sigma(x/\eta)$ is $\eta$-periodic. The method of H-convergence \cite{murattartar1978} (see also \cite{murat1997h}) can be used when the family $\sigma^\eta(x)$ is a general family of $L^\infty$ coefficients (see \cite{allaire2002shape} for a pedagogical exposition of the method of H-convergence). In the regime $\eta \ll \eps$, we need to homogenize the linear Boltzmann equation in the limit $\eta\to0$ for a fixed $\eps$ (see \cite{dumas2000homogenization} on the homogenization of linear transport equations). This shall be followed by a diffusion asymptotic in the $\eps\to0$ limit. These points have already been observed in an expository article by F. Golse \cite{golse1991particle}.

The present article addresses the issue of simultaneous limit procedure when both the small parameters $\eps$ and $\eta$ vanish. This problem has been addressed in \cite{sentis1980approximation, ben2005diffusion, abdallah2012diffusion, bal2012corrector} when the heterogeneous scattering coefficient is periodic. A first work in this direction goes back to the work of R. Sentis \cite{sentis1980approximation} where the heterogeneity length scale $\eta$ is related to the mean free path as $\eta=\eps^\beta$ with $\beta<1$. Recently, there has been a revival of this problem. The works \cite{abdallah2012diffusion, bal2012corrector} address this problem in the periodic setting and in the regime $\eps\ll\eta$. The approach in \cite{abdallah2012diffusion, bal2012corrector} is to introduce a new parameter $\frac{\eps}{\eta}$ and study some cell problems involving the new small parameter $\frac{\eps}{\eta}$. They extensively use the method of two-scale convergence. We also cite \cite{allaire1999homogenization} where the spectral problem associated with the scaled linear Boltzmann equation is studied when $\eta=\eps$. For the simultaneous limit procedures in the case $\eta=\eps$, we cite \cite{goudon2001approximation, goudon2003homogenization} (see also Chapter 7 in the lecture notes \cite{allaire2013transport}).

Our work, essentially, goes beyond the periodic setting in the spirit of the compensated compactness theory \cite{murattartar1978, murat1978compacite} developed by F. Murat and L. Tartar in the context of homogenization of elliptic and parabolic problems in the late 1970's. All the results and computations in this article are presented for a special case of the stationary linear Boltzmann, i.e. the probability density function is supposed to be time independent. All the results can be straightaway generalized to the time-dependent setting (see Remark \ref{Rem:Non-stationary setting}).

Similar to the work of R. Sentis \cite{sentis1980approximation}, we assume that the two scaling parameters are related as $\eta=\eps^\beta$ where $\beta\in(0,\infty)$. Note, however, that the results in \cite{sentis1980approximation} hold only when $\beta<1$ and under periodicity assumption on the heterogeneous scattering coefficient. Our main result in one dimensional setting (Theorem \ref{thm:1d}) is that the solutions to the linear Boltzmann equation converge to the solutions of an elliptic problem whenever $\beta\le 2$. We also obtain an explicit form of the effective diffusion coefficient in the limit equation -- given in terms of some velocity averages and the weak-* limit of the heterogeneous scattering coefficient. The one-dimensional setting is very special as the divergence operator coincides with the gradient operator resulting in uniform $H^1$ estimates for certain family of second moments -- refer to Section \ref{sec:1d} for further details.

Our main result in any arbitrary dimension is that if the heterogeneous coefficient satisfies
\begin{align*}
\left\|
\frac{\bar{\sigma}(x) - \sigma^\eps(x)}{\bar{\sigma}(x)}
\right\|_{H^{-\frac12}(\Omega)}
=
\mathcal{O}(\eps^{1+})
\end{align*}
for some $\bar{\sigma}\in L^\infty(\Omega)$ and where the family $\sigma^\eps(x)$ is nothing but the family $\sigma^\eta(x)$ with $\eta=\eps^\beta$, then the solutions to the linear Boltzmann equation converge to the solutions of a diffusion equation in physical space. More precisely, we show that the family
\[
f^{\eps,\eta}(x,v) \rightharpoonup \rho(x) \qquad \mbox{ weakly in }L^2(\Omega\times\mathcal{V};{\rm d}x\, {\rm d}\mu) \qquad \mbox{ as }\eps, \eta\to0
\]
where $\rho(x)$ is the $L^2$-weak limit of the family of local densities and that it is the unique solution to a diffusion equation. Our result in arbitrary dimension (Theorem \ref{thm:H-1/2}) essentially employs the moments approach inspired by \cite{bardos1991fluid} and uses the regularity of velocity averages guaranteed by the now well-known velocity averaging lemma \cite{golse1985resultat, golse1988regularity}.

\medskip

{\centering {\bf Plan of the paper:}} In Section \ref{sec:slbe}, we present the linear kinetic model, the scaling parameters and the scaling considered in this article. Section \ref{sec:uniform} gives some uniform (with respect to the scaling parameters $\eps$ and $\eta$) estimates on the solutions to the linear transport equation and associated velocity averages. In Section \ref{sec:moments}, we briefly explain the moments method as given in \cite{bardos1991fluid} and apply this method to the linear Boltzmann equation. Section \ref{sec:1d} is devoted to deriving the limit equation (in the $\eps\to0$ limit) in one dimensional setting -- Theorem \ref{thm:1d}. In Section \ref{sec:arbitrary}, we give a result in any arbitrary dimension under a certain assumption on the scattering coefficient. Finally, in Section \ref{sec:conclude} we give some concluding remarks and perspectives.

\medskip

\noindent{\bf Acknowledgments.} 
This work was initiated during C.B.'s visit to the University of Cambridge. It was supported by the ERC grant \textsc{matkit}.
The authors warmly thank R\'emi Sentis for his constructive suggestions during the preparation of this article.
The authors would also like to thank Thomas Holding for his valuable remarks on a preliminary version of this manuscript and Cl\'ement Mouhot for helpful suggestions.
H.H. acknowledges the support of the ERC grant MATKIT and the EPSRC programme grant ``Mathematical fundamentals of Metamaterials for multiscale Physics and Mechanic'' (EP/L024926/1).

\section{Stationary linear Boltzmann equation}\label{sec:slbe}

Let $f(x,v)$ be the distribution function which depends on $x\in\Omega\subset\R^d$ (space position) and $v\in\mathcal{V}$ (velocity). The distribution function models the probability density of mono-kinetic particles interacting with the background medium. The velocity space can be either of the following:
\begin{align*}
\mathcal{V} = \R^d;
\quad
\mathcal{V} = \S^{d-1};
\quad
\mathcal{V} = B(0,l):= \Big\{ v\in \R^d\, \mbox{ s.t. }|v|\le l\Big\}.
\end{align*}
We denote by $\mu$ a Borel probability measure on $\mathcal{V}$. We further suppose that
\begin{align}\label{eq:lbe:v-dmuv-0}
\int_{\mathcal{V}}
v
\, {\rm d}\mu(v)
= 0.
\end{align}
In order to define the boundary conditions, taking $n(x)$ to be the unit exterior normal to $\Omega$ at the point $x\in\partial\Omega$, we introduce the following notations:
\begin{align*}
& \Sigma := \lbrace (x,v) \in \partial\Omega\times\mathcal{V} \rbrace 
& \mbox{ Phase space Boundary,}
\\
& \Sigma_+ := \lbrace (x,v) \in \partial\Omega\times\mathcal{V}\mbox{ s.t. } v \cdot n(x) > 0 \rbrace
& \mbox{ Outgoing Boundary,}
\\
& \Sigma_- := \lbrace (x,v) \in \partial\Omega\times\mathcal{V}\mbox{ s.t. } v \cdot n(x) < 0 \rbrace
& \mbox{ Incoming Boundary,}
\\
& \Sigma_0 := \lbrace (x,v) \in \partial\Omega\times\mathcal{V}\mbox{ s.t. } v \cdot n(x) = 0 \rbrace
& \mbox{ Grazing set.}
\end{align*}
Denote by $\gamma_+ f$ (respectively $\gamma_- f$), the trace of $f$ on $\Sigma_+$ (respectively on $\Sigma_-$).\\
The goal is to perform the simultaneous limiting procedure for the scaled problem:
\begin{subequations}
\begin{align}
f^{\eps,\eta}(x,v)
+ \frac{1}{\eps} v \cdot \nabla_x f^{\eps,\eta}(x,v)
+ \frac{1}{\eps^2}\sigma^\eta\left(x\right) \mathcal{L} f^{\eps,\eta}(x,v)
= g(x)
\qquad
\mbox{ for }(x,v)\in \Omega\times\mathcal{V},\label{eq:lbe:model}
\\
\gamma_- f^{\eps,\eta}(x,v)
= 0
\qquad \qquad \qquad
\mbox{ for }(x,v)\in \Sigma_-,\label{eq:lbe:absorption}
\end{align}
\end{subequations}
where the linear Boltzmann operator is the following integral operator:
\begin{align}\label{eq:lbe:Boltzmann-operator}
\mathcal{L} g (v)
:=
g(v)
-
\int_{\mathcal{V}} g(w)
\, {\rm d}\mu(w)
\qquad
\mbox{ for any }g\in L^1(\mathcal{V};{\rm d}\mu).
\end{align}
The family of heterogeneous scattering coefficients indexed by $\eta$, i.e. $\sigma^\eta(x)$ is assumed to be a family of differentiable function such that there exist uniform (with respect to $\eta$) constants $\mathfrak{a}, \mathfrak{b}, \mathfrak{c}$ such that
\begin{align*}
& 0 < \mathfrak{a} \le \sigma^\eta(x) \le \mathfrak{b} < +\infty\quad \forall \, x\in\Omega
\\[0.1 cm]
\mbox{ and }& \|\nabla \sigma^\eta\|_{L^\infty(\Omega)} \le \mathfrak{c}\, \eta^{-1}.
\end{align*}
The source term in \eqref{eq:lbe:model} is assumed to be square-integrable in the space variable, i.e.
\begin{align*}
\|g\|_{L^2(\Omega)}\le C.
\end{align*}
We suppose that the heterogeneity length scale $\eta$ and the mean free path $\eps$ are related as
\begin{align*}
\eta = \eps^\beta
\qquad
\mbox{ for }\beta>0.
\end{align*}
Hence we drop the superscript $\eta$ in \eqref{eq:lbe:model} and index the family of heterogeneous scattering coefficients by $\eps$, i.e. $\sigma^\eps(x)$ which inherits the following bounds from $\sigma^\eta$ given above.
\begin{equation}\label{eq:lbe:hypo-cross-section}
\begin{array}{ll}
& 0 < \mathfrak{a} \le \sigma^\eps(x) \le \mathfrak{b} < +\infty\quad \forall \, x\in\Omega
\\[0.1 cm]
\mbox{ and }& \|\nabla \sigma^\eps\|_{L^\infty(\Omega)} \le \mathfrak{c}\, \eps^{-\beta}
\end{array}
\end{equation}
where the constants $\mathfrak{a}, \mathfrak{b}, \mathfrak{c}$ are uniform with respect to $\eps$. 
\begin{Rem}\label{rem:lbe:uniform}
Consider $\sigma\in W^{1,\infty}(\R^d)$ such that for a.e. $x\in\R^d$,
\begin{align*}
0 < \mathfrak{a} \le \sigma(x) \le \mathfrak{b} < \infty
\qquad\mbox{ and }\qquad
\|\nabla \sigma\|_{L^\infty(\R^d)}
\le \mathfrak{c} <\infty.
\end{align*}
Now consider the family
\begin{align*}
\sigma^\eps(x) := \sigma\left(\frac{x}{\eps^\beta}\right).
\end{align*}
Then, the assumptions given in \eqref{eq:lbe:hypo-cross-section} on the heterogeneous coefficient are satisfied by the above defined family of coefficients.
\end{Rem}
The main objective of this article is to study the following stationary problem in the $\eps\to0$ limit:
\begin{subequations}
\begin{align}
f^\eps(x,v)
+ \frac{1}{\eps} v \cdot \nabla_x f^\eps(x,v)
+ \frac{1}{\eps^2}\sigma^\eps(x) \mathcal{L} f^\eps(x,v)
= g(x)
\qquad
\mbox{ for }(x,v)\in \Omega\times\mathcal{V},\label{eq:lbe:model:eps}
\\
\gamma_- f^\eps(x,v)
= 0
\qquad \qquad \qquad
\mbox{ for }(x,v)\in \Sigma_-.\label{eq:lbe:absorption:eps}
\end{align}
\end{subequations}
We shall prove that the entire family $f^\eps(x,v)$ of solutions to the above kinetic equation exhibit the following compactness property:
\[
f^\eps(x,v) \rightharpoonup \rho(x)
\qquad \mbox{ weakly in }L^2(\Omega\times\mathcal{V};{\rm d}x\, {\rm d}\mu(v))
\]
where $\rho(x)$ is the $L^2$-weak limit of the associated local densities, i.e.
\[
\int_{\mathcal{V}} f^\eps(x,v)\, {\rm d}\mu(v)
\rightharpoonup \rho(x)
\qquad \mbox{ weakly in }L^2(\Omega)
\]
and the above weak limit uniquely solves a second order elliptic equation:
\[
\mathcal{A} \rho(x) = g(x)
\qquad \mbox{ in }\Omega.
\]
Further details on the elliptic operator $\mathcal{A}$ shall be given in Sections \ref{sec:1d} and \ref{sec:arbitrary}.

\section{Uniform A priori bounds}\label{sec:uniform}

In order to perform the asymptotic analysis in the $\eps\to0$ limit for \eqref{eq:lbe:model:eps}-\eqref{eq:lbe:absorption:eps}, we derive uniform (with respect to $\eps$) $L^2$-estimates on the solution family $f^\eps(x,v)$ and some associated velocity averages.\\
We first show that the Dirichlet form associated with the integral operator $\mathcal{L}$ in $L^2(\mathcal{V};{\rm d}\mu)$ is positive semi-definite.

\begin{Lem}\label{lem:Lff-ge0}
For any $\phi\in L^2(\mathcal{V};{\rm d}\mu)$, we have
\begin{align*}
\int_{\mathcal{V}} 
\phi(v) \mathcal{L} \phi(v)
\, {\rm d}\mu(v)
=
\frac{1}{2}
\iint_{\mathcal{V}\times\mathcal{V}}
\Big( \phi(v) - \phi(w) \Big)^2
\, {\rm d}\mu(w)\, {\rm d}\mu(v)
\ge 0.
\end{align*}
\end{Lem}

\begin{proof}
By the definition of the Boltzmann operator \eqref{eq:lbe:Boltzmann-operator}, we have:
\begin{align*}
\int_{\mathcal{V}} 
\phi(v) \mathcal{L} \phi(v)
\, {\rm d}\mu(v)
=
\int_{\mathcal{V}}
|\phi(v)|^2
\, {\rm d}\mu(v)
- \iint_{\mathcal{V}\times\mathcal{V}}
\phi(v) \phi(w)
\, {\rm d}\mu(w)\, {\rm d}\mu(v).
\end{align*}
Split the first integral on the right hand side of the above expression into two, thus yielding
\begin{align*}
&
\frac{1}{2}
\int_{\mathcal{V}}
|\phi(v)|^2
\, {\rm d}\mu(v)
+
\frac{1}{2}
\int_{\mathcal{V}}
|\phi(w)|^2
\, {\rm d}\mu(w)
- 
\iint_{\mathcal{V}\times\mathcal{V}}
\phi(v) \phi(w)
\, {\rm d}\mu(w)\, {\rm d}\mu(v).
\end{align*}
Thanks to $\mu$ being a probability density on $\mathcal{V}$, we can rewrite the above expression as
\begin{align*}
&
\frac{1}{2}
\iint_{\mathcal{V}\times\mathcal{V}}
\Big(
|\phi(v)|^2
+ |\phi(w)|^2
- 2 \phi(v) \phi(w)
\Big)
\, {\rm d}\mu(w)\, {\rm d}\mu(v)
\ge 0.
\end{align*}
Hence the result.
\end{proof}

Next, by using the energy approach (i.e. by choosing appropriate multiplier), we prove an entropy inequality associated with the stationary model \eqref{eq:lbe:model:eps}-\eqref{eq:lbe:absorption:eps}.

\begin{Lem}\label{lem:lbe:entropy-ineq}
The solution $f^\eps(x,v)$ to the linear Boltzmann equation \eqref{eq:lbe:model:eps}-\eqref{eq:lbe:absorption:eps} satisfies the following entropy inequality:
\begin{align*}
\int_\Omega\int_{\mathcal{V}}
|f^\eps(x,v)|^2
\, {\rm d}\mu(v)\, {\rm d}x
+
\frac{1}{\eps^2} \int_\Omega \iint_{\mathcal{V}\times\mathcal{V}}
\sigma^\eps(x) & \Big( f^\eps(x,v) - f^\eps(x,w) \Big)^2
\, {\rm d}\mu(w)\, {\rm d}\mu(v)\, {\rm d}x
\\
&
\le
\int_\Omega
|g(x)|^2
\, {\rm d}x.
\end{align*}
\end{Lem}

\begin{proof}
Multiply \eqref{eq:lbe:model:eps} by $f^\eps(x,v)$ and integrate over $\mathcal{V}$ yielding
\begin{align*}
\int_{\mathcal{V}}
|f^\eps(x,v)|^2
\, {\rm d}\mu(v)
+ \frac{1}{2\eps}
\int_{\mathcal{V}}
v \cdot \nabla_x |f^\eps(x,v)|^2
\, {\rm d}\mu(v)
& + \frac{\sigma^\eps(x)}{\eps^2}
\int_{\mathcal{V}}
f^\eps(x,v) \mathcal{L} f^\eps(x,v)
\, {\rm d}\mu(v)
\\
& = 
\int_{\mathcal{V}}
g(x) f^\eps(x,v)
\, {\rm d}\mu(v).
\end{align*}
Using Lemma \ref{lem:Lff-ge0} for the Dirichlet form and integrating over $\Omega$ yields:
\begin{align*}
& \int_\Omega\int_{\mathcal{V}}
|f^\eps(x,v)|^2
\, {\rm d}\mu(v)\, {\rm d}x
+ \frac{1}{2\eps}
\int_\Omega\int_{\mathcal{V}}
v \cdot \nabla_x |f^\eps(x,v)|^2
\, {\rm d}\mu(v)\, {\rm d}x
\\
& + \int_\Omega\frac{\sigma^\eps(x)}{2\eps^2}
\iint_{\mathcal{V}\times\mathcal{V}}
\Big( f^\eps(x,v) - f^\eps(x,w) \Big)^2
\, {\rm d}\mu(w){\rm d}\mu(v){\rm d}x
= 
\int_\Omega\int_{\mathcal{V}}
g(x) f^\eps(x,v)
\, {\rm d}\mu(v)\, {\rm d}x.
\end{align*}
Consider the transport term in the previous expression and perform an integration by parts in the space variable yielding (with ${\rm d}\Gamma(x)$ as the surface measure on $\partial\Omega$):
\begin{align*}
\int_\Sigma
\left( v \cdot n(x) \right) |\gamma f^\eps(x,v)|^2
\, {\rm d}\mu(v)\, {\rm d}\Gamma(x)
=
\int_{\Sigma_+}
\left( v \cdot n(x)\right) |\gamma_+ f^\eps(x,v)|^2
\, {\rm d}\mu(v)\, {\rm d}\Gamma(x)
\end{align*}
because of the zero absorption boundary condition \eqref{eq:lbe:absorption:eps} on the incoming phase-space boundary. In the right hand side of the above expression, make the change of variables: $v\to v_* :=\mathcal{R}_x v$ for each $x\in\partial\Omega$, where $\mathcal{R}_x v = v - 2(v\cdot n(x))n(x)$ is the reflection operator, yielding:
\begin{align*}
\int_{\Sigma_+}
\left( v \cdot n(x) \right) |\gamma_+ f^\eps(x,v)|^2
\, {\rm d}v\, {\rm d}\Gamma(x)
=
- \int_{\Sigma_-}
\left( v_* \cdot n(x) \right) |\gamma_- f^\eps(x,v_*)|^2
\, {\rm d}\mu(v_*)\, {\rm d}\Gamma(x)
=
0
\end{align*}
again by the absorption boundary condition \eqref{eq:lbe:absorption:eps}. Hence the transport term does not contribute. We are left with the following expression:
\begin{align*}
\int_\Omega\int_{\mathcal{V}}
|f^\eps(x,v)|^2
\, {\rm d}\mu(v)\, {\rm d}x
& + \int_\Omega\frac{\sigma^\eps(x)}{2\eps^2}
\iint_{\mathcal{V}\times\mathcal{V}}
\Big( f^\eps(x,v) - f^\eps(x,w) \Big)^2
\, {\rm d}\mu(w)\, {\rm d}\mu(v)\, {\rm d}x
\\
& = 
\int_\Omega\int_{\mathcal{V}}
g(x) f^\eps(x,v)
\, {\rm d}\mu(v)\, {\rm d}x.
\end{align*}
Apply Cauchy-Schwarz inequality and Young's inequality to the term involving the source term:
\begin{align*}
\int_\Omega\int_{\mathcal{V}}
g(x) f^\eps(x,v)
\, {\rm d}\mu(v)\, {\rm d}x
& \le 
\left(
\int_\Omega\int_{\mathcal{V}}
|g(x)|^2
\, {\rm d}\mu(v)\, {\rm d}x
\right)^{1/2}
\left(
\int_\Omega\int_{\mathcal{V}}
|f^\eps(x,v)|^2
\, {\rm d}\mu(v)\, {\rm d}x
\right)^{1/2}
\\
& \le 
\frac12 \int_\Omega
|g(x)|^2
\, {\rm d}x
+
\frac12 \int_\Omega\int_{\mathcal{V}}
|f^\eps(x,v)|^2
\, {\rm d}\mu(v)\, {\rm d}x.
\end{align*}
Using this inequality in the previous equality yields the entropy inequality.
\end{proof}

Our next result is to derive some uniform $L^2$-estimates using the entropy inequality. We use the following notation:
\begin{align*}
\langle \mathbf{h} \rangle 
:= 
\int_{\mathcal{V}} \mathbf{h}(v)
\, {\rm d}\mu(v)
\qquad
\mbox{ for all }
\mathbf{h} \in L^1(\mathcal{V};{\rm d}\mu).
\end{align*}

\begin{Lem}\label{lem:uniform-apriori}
Let $f^\eps(x,v)$ be the solution to \eqref{eq:lbe:model:eps}-\eqref{eq:lbe:absorption:eps}. We have the following estimates:
\begin{subequations}
\begin{align}
& \|f^\eps\|_{L^2(\Omega\times\mathcal{V};{\rm d}x{\rm d}\mu)} \le \|g\|_{L^2(\Omega)}\label{eq:lbe:uniform:f-L2-x-v}\\[0.2 cm]
& \|(f^\eps - \langle f^\eps \rangle)\|_{L^2(\Omega\times\mathcal{V};{\rm d}x{\rm d}\mu)} \le \frac{\eps}{\sqrt{\mathfrak{a}}} \|g\|_{L^2(\Omega)}\label{eq:lbe:uniform:f-lfr-L2-x-v}\\[0.2 cm]
& \|\langle f^\eps \rangle\|_{L^2(\Omega)} \le \|g\|_{L^2(\Omega)}\label{eq:lbe:uniform:lfr-L2-x}
\end{align}
\end{subequations}
\end{Lem}

\begin{proof}
The uniform estimate \eqref{eq:lbe:uniform:f-L2-x-v} follows directly from the entropy inequality (Lemma \ref{lem:lbe:entropy-ineq}).\\
Next, we focus on the uniform estimate \eqref{eq:lbe:uniform:f-lfr-L2-x-v}. We have
\begin{align*}
&
\|f^\eps(x,\cdot) - \langle f^\eps \rangle (x)\|^2_{L^2(\mathcal{V};{\rm d}\mu)}
= 
\int_{\mathcal{V}}
\left(
\int_{\mathcal{V}} \Big( f^\eps(x,v) - f^\eps(x,w) \Big) \, {\rm d}\mu(w)
\right)^2
\, {\rm d}\mu(v)
\end{align*}
Applying Cauchy-Schwarz we get:
\begin{align*}
\int_{\mathcal{V}} \Big( f^\eps(x,v) - f^\eps(x,w) \Big) \, {\rm d}\mu(w)
\le 
\left(
\int_{\mathcal{V}} \Big( f^\eps(x,v) - f^\eps(x,w) \Big)^2 \, {\rm d}\mu(w)
\right)^{1/2}
\end{align*}
because $\mu$ is a probability measure on $\mathcal{V}$. Hence we have
\begin{align*}
\int_\Omega
\|f^\eps(x,\cdot) - \langle f^\eps \rangle (x)\|^2_{L^2(\mathcal{V};{\rm d}\mu)}
\, {\rm d}x
& \le 
\int_\Omega \iint_{\mathcal{V}\times\mathcal{V}}
\Big( f^\eps(x,v) - f^\eps(x,w) \Big)^2 
\, {\rm d}\mu(w)\, {\rm d}\mu(v)\, {\rm d}x
\\
& \le \frac{\eps^2}{\mathfrak{a}} \|g\|^2_{L^2(\Omega)}
\end{align*}
where the last inequality is because of the entropy inequality (Lemma \ref{lem:lbe:entropy-ineq}).\\
Our next goal is to prove the uniform estimate \eqref{eq:lbe:uniform:lfr-L2-x}. Consider
\begin{align*}
\|\langle f^\eps \rangle\|^2_{L^2(\Omega)}
=
\int_\Omega
\left(
\int_{\mathcal{V}} f^\eps(x,v)
\, {\rm d}\mu(v)
\right)^2
\, {\rm d}x.
\end{align*}
By Cauchy-Schwarz inequality, we have:
\begin{align*}
\left(
\int_{\mathcal{V}} f^\eps(x,v)
\, {\rm d}\mu(v)
\right)^2
\le 
\left(
\int_{\mathcal{V}} |f^\eps(x,v)|^2
\, {\rm d}\mu(v)
\right)
\left(
\int_{\mathcal{V}} 1^2
\, {\rm d}\mu(v)
\right)
=
\int_{\mathcal{V}} |f^\eps(x,v)|^2
\, {\rm d}\mu(v)
\end{align*}
because $\mu$ is a probability measure on $\mathcal{V}$. Thus we have:
\begin{align*}
\|\langle f^\eps \rangle\|_{L^2(\Omega)}
\le 
\|f^\eps\|_{L^2(\Omega\times\mathcal{V};{\rm d}x{\rm d}\mu)}
\le 
\|g\|_{L^2(\Omega)}.
\end{align*}
Hence the result.
\end{proof}
Next, we prove a crucial estimate on the velocity average $\langle v f^\eps \rangle$. To begin with, we observe that the integral operator $\mathcal{L}$ is self-adjoint in $L^2(\mathcal{V};{\rm d}\mu)$, i.e.
\begin{align*}
\int_{\mathcal{V}}
\psi(v) \mathcal{L}\phi(v)
\, {\rm d}\mu(v)
= 
\int_{\mathcal{V}}
\phi(v) \mathcal{L}\psi(v)
\, {\rm d}\mu(v)
\qquad
\mbox{ for all }\phi,\psi\in L^2(\mathcal{V};{\rm d}\mu).
\end{align*}
This observation follows from the following successive equalities for the inner product in $L^2(\mathcal{V};{\rm d}\mu)$:
\begin{align*}
\int_{\mathcal{V}}
\psi(v) \mathcal{L}\phi(v)
\, {\rm d}\mu(v)
& = 
\int_{\mathcal{V}}
\psi(v)\phi(v)
\, {\rm d}\mu(v)
- 
\int_{\mathcal{V}}
\psi(v)
\int_{\mathcal{V}}
\phi(w)
\, {\rm d}\mu(w)\, {\rm d}\mu(v)
\\
& = 
\int_{\mathcal{V}}
\phi(v)\psi(v)
\, {\rm d}\mu(v)
- 
\int_{\mathcal{V}}
\phi(v)
\int_{\mathcal{V}}
\psi(w)
\, {\rm d}\mu(w)\, {\rm d}\mu(v)
\\
& =
\int_{\mathcal{V}}
\phi(v) \mathcal{L}\psi(v)
\, {\rm d}\mu(v).
\end{align*}
Next, we give a very important representation for the velocity variable in terms of the integral operator $\mathcal{L}$.
\begin{Lem}\label{lem:v-rep}
For each $i$-th component of the velocity variable, we have:
\begin{align}
v_i = \mathcal{L}v_i.
\end{align}
\end{Lem}
\begin{proof}
Take $\phi(v) = v_i$ and apply the integral operator $\mathcal{L}$ on to $\phi$, i.e.
\begin{align*}
\mathcal{L}v_i
= 
v_i 
- 
\int_{\mathcal{V}}
w_i
\, {\rm d}\mu(w)
=
v_i
\end{align*}
thanks to our assumption \eqref{eq:lbe:v-dmuv-0}.
\end{proof}

\begin{Lem}\label{lem:crucial}
Let $f^\eps(x,v)$ be the solution to \eqref{eq:lbe:model:eps}-\eqref{eq:lbe:absorption:eps}. We have the following estimate:
\begin{align}\label{eq:curcial-estimate}
\|\langle v f^\eps \rangle\|_{[L^2(\Omega)]^d}
\le 
\eps \sqrt{\langle |v|^2 \rangle} \|g\|_{L^2(\Omega)}.
\end{align}
\end{Lem}

\begin{proof}
Consider the velocity average:
\begin{align*}
\frac{1}{\eps}\langle v f^\eps \rangle
= 
\frac{1}{\eps}
\int_{\mathcal{V}}
v f^\eps(x,v)
\, {\rm d}\mu(v).
\end{align*}
Crucial argument is to substitute for the velocity variable in the above expression using Lemma \ref{lem:v-rep} which yields:
\begin{align*}
\frac{1}{\eps}\langle v f^\eps \rangle
= 
\frac{1}{\eps}
\int_{\mathcal{V}}
f^\eps(x,v) \mathcal{L}v
\, {\rm d}\mu(v)
= 
\frac{1}{\eps}
\int_{\mathcal{V}}
v \mathcal{L}f^\eps(x,v)
\, {\rm d}\mu(v)
\end{align*}
because $\mathcal{L}$ is self-adjoint. Substituting for the Boltzmann operator in the above expression, we get
\begin{align*}
\frac{1}{\eps}\langle v f^\eps \rangle
& =
\iint_{\mathcal{V}\times\mathcal{V}}
v
\frac{1}{\eps}
\Big(
f^\eps(x,v) - f^\eps(x,w)
\Big)
\, {\rm d}\mu(w)\, {\rm d}\mu(v)
\\
& \le 
\left(
\iint_{\mathcal{V}\times\mathcal{V}}
|v|^2
\, {\rm d}\mu(w)\, {\rm d}\mu(v)
\right)^{1/2}
\left(
\iint_{\mathcal{V}\times\mathcal{V}}
\frac{1}{\eps^2}
\Big(
f^\eps(x,v) - f^\eps(x,w)
\Big)^2
\, {\rm d}\mu(w)\, {\rm d}\mu(v)
\right)^{1/2}
\end{align*}
where Cauchy-Schwarz inequality is used. Squaring the above inequality, we get
\begin{align*}
\left|
\frac{1}{\eps}\langle v f^\eps \rangle
\right|^2
\le 
\left\langle |v|^2 \right\rangle
\iint_{\mathcal{V}\times\mathcal{V}}
\frac{1}{\eps^2}
\Big(
f^\eps(x,v) - f^\eps(x,w)
\Big)^2
\, {\rm d}\mu(w)\, {\rm d}\mu(v).
\end{align*}
Integrate the above inequality on $\Omega$ yielding:
\begin{align*}
\int_\Omega
\left|
\frac{1}{\eps}\langle v f^\eps \rangle
\right|^2
\, {\rm d}x
& \le 
\left\langle |v|^2 \right\rangle
\int_\Omega\iint_{\mathcal{V}\times\mathcal{V}}
\frac{1}{\eps^2}
\Big(
f^\eps(x,v) - f^\eps(x,w)
\Big)^2
\, {\rm d}\mu(w)\, {\rm d}\mu(v)\, {\rm d}x
\\
& \le 
\left\langle |v|^2 \right\rangle
\|g\|^2_{L^2(\Omega)}
\end{align*}
where we have used the entropy inequality (Lemma \ref{lem:lbe:entropy-ineq}), thus proving the crucial estimate.
\end{proof}

Next, we prove uniform estimates on $\langle (v \otimes v)f^\eps\rangle$ and its divergence.

\begin{Lem}\label{lem:H-div}
Let $f^\eps(x,v)$ be the solution to \eqref{eq:lbe:model:eps}-\eqref{eq:lbe:absorption:eps}. We have the following estimates:
\begin{subequations}
\begin{align}
& \left\|
\langle
(v \otimes v)f^\eps
\rangle
\right\|_{[L^2(\Omega)]^{d\times d}}
\le 
\sqrt{\left\langle|v\otimes v|^2\right\rangle}
\|g\|_{L^2(\Omega)},\label{eq:v-tensor-v-f-L2}\\[0.2 cm]
& \left\|
\nabla_x \cdot 
\langle
(v \otimes v)f^\eps
\rangle
\right\|_{[L^2(\Omega)]^d}
\le 
\mathfrak{b} \|g\|_{L^2(\Omega)}.\label{eq:nabla-v-tensor-v-f-L2}
\end{align}
\end{subequations}
\end{Lem}

\begin{proof}
Consider
\begin{align*}
\| \langle (v\otimes v)f^\eps \rangle \|^2_{[L^2(\Omega)]^{d\times d}}
=
\int_\Omega
\left(
\int_{\mathcal{V}}
\Big(v\otimes v\Big)f^\eps(x,v)
\, {\rm d}\mu(v) 
\right)^2
\, {\rm d}x.
\end{align*}
Cauchy-Schwarz inequality yields:
\begin{align*}
\left(
\int_{\mathcal{V}}
\Big(v\otimes v\Big)f^\eps(x,v)
\, {\rm d}\mu(v) 
\right)^2
\le 
\left\langle
|v\otimes v|^2
\right\rangle
\left(
\int_{\mathcal{V}}
|f^\eps(x,v)|^2
\, {\rm d}\mu(v)
\right).
\end{align*}
Integrating the above inequality over $\Omega$ yields:
\begin{align*}
\| \langle (v\otimes v)f^\eps \rangle \|^2_{[L^2(\Omega)]^{d\times d}}
\le 
\left\langle
|v\otimes v|^2
\right\rangle
\int_\Omega \int_{\mathcal{V}}
|f^\eps(x,v)|^2
\, {\rm d}\mu(v)\, {\rm d}x.
\end{align*}
Using the uniform estimate \eqref{eq:lbe:uniform:f-L2-x-v} from Lemma \ref{lem:uniform-apriori} yields the estimate \eqref{eq:v-tensor-v-f-L2}.\\
Next, we focus on the estimate \eqref{eq:nabla-v-tensor-v-f-L2}. To that end, multiply the stationary problem \eqref{eq:lbe:model:eps} by the $i$-th component of the velocity variable and integrate over $\mathcal{V}$ yielding:
\begin{align*}
\eps \langle v_i f^\eps \rangle
+ \sum_{j=1}^d \left\langle v_i v_j \frac{\partial f^\eps}{\partial x_j} \right\rangle
+ \frac{\sigma^\eps(x)}{\eps} \langle v_i f^\eps \rangle
= 0.
\end{align*}
The scattering coefficient is bounded in $L^\infty(\Omega)$. Hence the crucial estimate \eqref{eq:curcial-estimate} of Lemma \ref{lem:crucial} would straightaway imply the following:
\begin{align*}
\mbox{for each }i\in\{1,\cdots, d\}
\qquad
\left\|
\sum_{j=1}^d \left\langle v_i v_j \frac{\partial f^\eps}{\partial x_j} \right\rangle
\right\|_{L^2(\Omega)}
\le 
\mathfrak{b} \|g\|_{L^2(\Omega)},
\end{align*}
thus proving the estimate \eqref{eq:nabla-v-tensor-v-f-L2}.
\end{proof}

\begin{Rem}\label{rem:H-div}
The result of Lemma \ref{lem:H-div} says that the matrix valued function $\langle (v\otimes v) f^\eps \rangle$ is in the Hilbert space $[H({\rm div};\Omega)]^d$, i.e. each row vector of $\langle (v\otimes v) f^\eps \rangle$ belongs to
\begin{align}\label{eq:def-H-div}
H({\rm div};\Omega)
:=
\left\{
u\in [L^2(\Omega)]^d
\mbox{ such that }
\nabla_x\cdot u \in L^2(\Omega)
\right\}.
\end{align}
\end{Rem}

\section{Moments method}\label{sec:moments}

We present a moments based approach to derive the limit behavior for \eqref{eq:lbe:model:eps}-\eqref{eq:lbe:absorption:eps} as $\eps\to0$. This method is essentially borrowed from \cite{bardos1991fluid}. For readers' convenience, we shall present this approach step-by-step.\\
{\bf Step I:} Integrate \eqref{eq:lbe:model:eps} over $\mathcal{V}$:
\begin{align}\label{eq:moment:step1}
\langle f^\eps \rangle
+ \frac{1}{\eps} \nabla \cdot \langle v f^{\eps}\rangle
= g(x).
\end{align}
{\bf Step II:} Multiply \eqref{eq:lbe:model:eps} by the velocity variable $v$ and integrate over $\mathcal{V}$:
\begin{align}\label{eq:moment:step2}
\frac{\eps}{\sigma^\eps(x)} \langle v f^\eps \rangle
+ \frac{1}{\sigma^\eps(x)}\nabla \cdot \langle (v\otimes v)f^{\eps}\rangle 
+ \frac{1}{\eps} \langle v f^{\eps}\rangle 
= 0.
\end{align}
{\bf Step III:} Multiply \eqref{eq:moment:step1} by a test function $\varphi(x)\in H^1_0(\Omega)$ and integrate over $\Omega$:
\begin{align}\label{eq:moment:step3}
\frac{1}{\eps}\int_\Omega \langle v f^{\eps}\rangle \cdot \nabla\varphi(x)\, {\rm d}x
=
\int_\Omega \langle f^\eps \rangle(x) \varphi(x)\, {\rm d}x
- \int_\Omega g(x)\varphi(x)\, {\rm d}x.
\end{align}
{\bf Step IV:} Take dot product of the vector equation \eqref{eq:moment:step2} with $\nabla\varphi(x)$ and integrate over $\Omega$:
\begin{equation}\label{eq:moment:step4}
\begin{array}{ll}
\displaystyle 
\int_\Omega \left( \frac{1}{\sigma^\eps(x)}\nabla \cdot \langle (v\otimes v)f^{\eps}\rangle \right)\cdot \nabla\varphi(x)\, {\rm d}x
& \displaystyle  + \frac{1}{\eps} \int_\Omega \langle v f^{\eps}\rangle \cdot \nabla\varphi(x)\, {\rm d}x
\\[0.4 cm]
& \displaystyle 
+ \eps \int_\Omega \frac{1}{\sigma^\eps(x)} \langle v f^\eps \rangle (x) \cdot \nabla\varphi(x)\, {\rm d}x
= 0.
\end{array}
\end{equation}
Using \eqref{eq:moment:step3} in \eqref{eq:moment:step4} yields the following expression in which we need to pass to the limit.
\begin{equation}\label{eq:moment:step5}
\begin{array}{ll}
\displaystyle 
\int_\Omega \left( \frac{1}{\sigma^\eps(x)}\nabla \cdot \langle (v\otimes v)f^{\eps}\rangle \right)\cdot \nabla\varphi(x)\, {\rm d}x
= \int_\Omega g(x)\varphi(x){\rm d}x
\\[0.4 cm]
\displaystyle 
- \int_\Omega \langle f^\eps \rangle(x) \varphi(x)\, {\rm d}x
- \eps \int_\Omega \frac{1}{\sigma^\eps(x)} \langle v f^\eps \rangle (x) \cdot \nabla\varphi(x)\, {\rm d}x.
\end{array}
\end{equation}
The \emph{moments method} culminates in passing to the limit as $\eps\to0$ in \eqref{eq:moment:step5} using some compactness properties of the family $f^\eps(x,v)$. In this article, this final step of the moments method is achieved in Section \ref{sec:1d} for the one-dimensional case and in Section \ref{sec:arbitrary} for any arbitrary dimension. Observe that the expression \eqref{eq:moment:step5} can be treated as a weak formulation for the following second-order differential equation in the $x$ variable:
\begin{align}\label{eq:moment:div-problem}
- \nabla_x \cdot \left( \frac{1}{\sigma^\eps(x)}\nabla \cdot \langle (v\otimes v)f^{\eps}\rangle  \right)
= \mathcal{G}^\eps(x),
\end{align}
where $\mathcal{G}^\eps:\Omega\to\R$ is defined as
\begin{align}\label{eq:moment:rhs}
\mathcal{G}^\eps (x) := g(x) - \langle f^\eps \rangle(x) + \eps \nabla_x\cdot \left( \frac{1}{\sigma^\eps(x)} \langle v f^\eps \rangle (x)\right).
\end{align}
Next, we give a result showing that the family $\mathcal{G}^\eps(x)$ is uniformly bounded in $L^2(\Omega)$ provided the exponent $\beta$ takes values in certain interval of $[0,\infty)$.
\begin{Lem}\label{lem:mathcal-G-uniform-L2}
Let $f^\eps(x,v)$ be the solution to \eqref{eq:lbe:model:eps}-\eqref{eq:lbe:absorption:eps} and suppose the exponent $\beta\le2$. Let $\mathcal{G}^\eps(x)$ be the family of scalar functions defined by \eqref{eq:moment:rhs}. There exists a constant $C$, independent of $\eps$, such that
\begin{align*}
\|\mathcal{G}^\eps\|_{L^2(\Omega)} \le C.
\end{align*}
\end{Lem}

\begin{proof}
By chain rule, we have
\begin{align}\label{eq:mathcal-G-eps}
\mathcal{G}^\eps(x) = g(x) - \langle f^\eps \rangle(x) + \frac{\eps}{\sigma^\eps(x)} \nabla_x\cdot \langle v f^\eps \rangle (x) - \eps \frac{\nabla_x \sigma^\eps(x)}{[\sigma^\eps(x)]^2}\cdot \langle v f^\eps \rangle(x).
\end{align}
Substitute for $\nabla_x\cdot \langle v f^\eps \rangle$ using the continuity equation \eqref{eq:moment:step1} in the above expression yielding:
\begin{align*}
\mathcal{G}^\eps(x) & = g(x) - \langle f^\eps \rangle(x) + \frac{\eps^2}{\sigma^\eps(x)} g(x) - \frac{\eps^2}{\sigma^\eps(x)}\langle f^\eps \rangle (x) - \eps \frac{\nabla_x \sigma^\eps(x)}{[\sigma^\eps(x)]^2}\cdot \langle v f^\eps \rangle(x)\\
& = \left( 1 + \frac{\eps^2}{\sigma^\eps(x)} \right) \Big( g(x) - \langle f^\eps \rangle (x) \Big)
- \eps \frac{\nabla_x \sigma^\eps(x)}{[\sigma^\eps(x)]^2}\cdot \langle v f^\eps \rangle(x).
\end{align*}
Computing the $L^2$-norm for $\mathcal{G}^\eps$ we have:
\begin{align*}
\|\mathcal{G}^\eps\|_{L^2(\Omega)} 
& \le 
\left\|
\left( 1 + \frac{\eps^2}{\sigma^\eps(x)} \right)
\right\|_{L^\infty}
\left(
\|g\|_{L^2(\Omega)} + \|\langle f^\eps\rangle\|_{L^2(\Omega)}
\right)
+ \eps
\left\|
\frac{\nabla\sigma^\eps}{[\sigma^\eps]^2}
\right\|_{L^\infty}
\|\langle v f^\eps\rangle\|_{[L^2(\Omega)]^d}\\
& \le 
C \left( 1 + \frac{1}{\mathfrak{a}} \right)
\left(
\|g\|_{L^2(\Omega)} + \|\langle f^\eps\rangle\|_{L^2(\Omega)}
\right)
+ C \left( \frac{\mathfrak{c}}{\mathfrak{a}^2}\right)
\eps^{1-\beta}
\|\langle v f^\eps\rangle\|_{[L^2(\Omega)]^d}
\end{align*}
where we have used the growth assumption \eqref{eq:lbe:hypo-cross-section} on the heterogeneous coefficient $\sigma^\eps$. The assumption on the exponent $\beta$ (i.e. $\beta\le2$) and the uniform estimates \eqref{eq:lbe:uniform:lfr-L2-x} on $\langle f^\eps \rangle$ help us arrive at the uniform $L^2$-bound for $\mathcal{G}^\eps(x)$.
\end{proof}

\section{One dimensional setting}\label{sec:1d}

In this section, we treat a special setting: both the spatial and velocity domains are one-dimensional, i.e. $x\in (-\ell, +\ell)$ and $v\in\mathcal{V}$ where $\mathcal{V}$ is either $\R$ or $(-1,+1)$. We consider the transport equation for the one particle distribution function $f^\eps(x,v)$:

\begin{subequations}
\begin{align}
f^\eps + \frac{1}{\eps}v \frac{{\rm d}f^\eps}{{\rm d}x} + \frac{\sigma^\eps(x)}{\eps^2} \Big( f^\eps - \langle f^\eps \rangle \Big) = g(x)
\quad
\mbox{ for }(x,v)\in (-\ell, +\ell) \times \mathcal{V},\label{eq:1d:model}\\
f^\eps(x,v) = 0 \quad \mbox{ for }x = \pm\ell \mbox{ and }v\in\mathcal{V},\label{eq:1d:absorption}
\end{align}
\end{subequations}
where the heterogeneous coefficient $\sigma^\eps(x)$ are assumed to satisfy the same regularity assumptions as before, i.e. \eqref{eq:lbe:hypo-cross-section}.

\begin{Thm}\label{thm:1d}
The solution family $f^\eps(x,v)$ to the one-dimensional stationary linear Boltzmann equation \eqref{eq:1d:model}-\eqref{eq:1d:absorption} exhibits the following compactness property:
\begin{align*}
f^\eps(x,v) \rightharpoonup \rho(x)
\qquad
\mbox{ weakly in }L^2((-\ell, +\ell)\times\mathcal{V};{\rm d}x{\rm d}\mu),
\end{align*}
where $\rho(x)$ is the unique solution to the stationary diffusion equation:
\begin{subequations}
\begin{align}
\rho(x) - \frac{{\rm d}}{{\rm d}x} \left(\frac{\mathfrak{D}}{\sigma^*}\frac{{\rm d}\rho}{{\rm d}x}\right) = g(x)
\qquad
\mbox{ for }x\in (-\ell, +\ell),\label{eq:1d:diff-limit-equation}
\\
\rho(x)=0
\qquad \qquad
\mbox{ for }x=\pm\ell,\label{eq:1d:diff-limit-BC}
\end{align}
\end{subequations}
where $\mathfrak{D}$ is a constant equal to $\langle v^2 \rangle$ and $\sigma^*$ is the $L^\infty$ weak-$^*$ limit of the sequence $\sigma^\eps$.
\end{Thm}

\begin{proof}
The a priori estimates of Lemma \ref{lem:H-div} in the one-dimensional setting imply that the sequence $\langle v^2 f^\eps \rangle$ is uniformly bounded in $H^1(-\ell, +\ell)$. Hence we can extract a sub-sequence such that
\begin{align}
\langle v^2 f^\eps \rangle \to \mathfrak{D} \rho
\qquad
\mbox{ strongly in }L^2(-\ell, +\ell),\label{eq:1d:strong-lim-l-v2f-eps-r}
\\
\frac{{\rm d}}{{\rm d}x}\langle v^2 f^\eps \rangle \rightharpoonup \mathfrak{D}\frac{{\rm d}\rho}{{\rm d}x}
\qquad
\mbox{ weakly in }L^2(-\ell, +\ell),\label{eq:1d:strong-lim-derive-l-v2f-eps-r}
\end{align}
where $\mathfrak{D}$ is a constant given by $\langle v^2 \rangle$ and $\rho(x)$ is the $L^2$ weak limit of the local densities, i.e.
\[
\int_{\mathcal{V}} f^\eps(x,v)\, {\rm d}\mu(v) \rightharpoonup \rho(x)
\qquad \mbox{ weakly in }L^2(-\ell, +\ell).
\]
The second order differential equation (i.e. the one similar to \eqref{eq:moment:div-problem}) that we get via the moments method (see Section \ref{sec:moments}) in this one-dimensional setting is the following:
\begin{align}\label{eq:1d:div-problem}
-\frac{{\rm d}}{{\rm d}x}
\left(
\frac{1}{\sigma^\eps(x)} \frac{{\rm d}\langle v^2 f^\eps\rangle}{{\rm d}x}
\right)
= 
\mathcal{G}^\eps(x),
\end{align}
where $\mathcal{G}^\eps(x)$ is given by
\begin{align*}
\mathcal{G}^\eps (x) := g(x) - \langle f^\eps \rangle(x) + \eps \frac{{\rm d}}{{\rm d}x}\left( \frac{1}{\sigma^\eps(x)} \langle v f^\eps \rangle (x)\right).
\end{align*}
Define
\begin{align}\label{eq:1d:def-zeta-eps}
\zeta^\eps(x)
:=
\frac{1}{\sigma^\eps(x)}\frac{{\rm d}}{{\rm d}x}\langle v^2 f^\eps \rangle.
\end{align}
We have the following uniform $L^2$-bound:
\begin{align}\label{eq:1d:uniform-L2-zeta-eps}
\|\zeta^\eps\|_{L^2(-\ell, +\ell)}
\le 
\frac{1}{\mathfrak{a}}
\left\|
\frac{{\rm d}}{{\rm d}x}\langle v^2 f^\eps \rangle
\right\|_{L^2(-\ell, +\ell)}
\le 
\frac{\mathfrak{b}}{\mathfrak{a}}
\|g\|_{L^2(-\ell, +\ell)}
\end{align}
where we have used the assumption \eqref{eq:lbe:hypo-cross-section} that $\sigma^\eps$ is bounded from below and the uniform a priori bound \eqref{eq:nabla-v-tensor-v-f-L2} from Lemma \ref{lem:H-div}. From \eqref{eq:1d:div-problem}, it follows that $\zeta^\eps(x)$ solves
\begin{align}\label{eq:1d:equation-for-zeta-eps}
- \frac{{\rm d}}{{\rm d}x} \zeta^\eps(x)
= 
\mathcal{G}^\eps(x).
\end{align}
The uniform $L^2$-estimate on $\mathcal{G}^\eps$ (Lemma \ref{lem:mathcal-G-uniform-L2}) yields a uniform bound on the derivative of $\zeta^\eps$:
\begin{align}\label{eq:1d:uniform-L2-derive-zeta-eps}
\left\|
\frac{{\rm d}\zeta^\eps}{{\rm d}x}
\right\|_{L^2(-\ell, +\ell)}
= 
\|\mathcal{G}^\eps\|_{L^2(-\ell, +\ell)}
\le C.
\end{align}
The estimates \eqref{eq:1d:uniform-L2-zeta-eps} and \eqref{eq:1d:uniform-L2-derive-zeta-eps} together imply that the sequence $\zeta^\eps$ is uniformly bounded in $H^1(\Omega)$. The compact embedding $H^1(-\ell, +\ell)\hookrightarrow L^2(-\ell, +\ell)$ implies that we can extract a sub-sequence and there exists a limit $\zeta^0$ such that
\begin{align}\label{eq:1d:strong-lim-zeta-0}
\zeta^\eps \to \zeta^0
\qquad
\mbox{ strongly in }L^2(-\ell, +\ell).
\end{align}
By the definition \eqref{eq:1d:def-zeta-eps} of $\zeta^\eps$, we have:
\begin{align*}
\sigma^\eps(x) \zeta^\eps(x) = \frac{{\rm d}}{{\rm d}x}\langle v^2 f^\eps \rangle
\end{align*}
Because of the strong convergence in \eqref{eq:1d:strong-lim-zeta-0}, we have the following:
\begin{align*}
\sigma^\eps \zeta^\eps \rightharpoonup \sigma^* \zeta^0
\qquad\mbox{ weakly in }L^2(-\ell, +\ell),
\end{align*}
where $\sigma^*$ is the $L^\infty$ weak-${}^*$ limit of the sequence $\sigma^\eps$. Identifying the above weak limit with the weak limit \eqref{eq:1d:strong-lim-derive-l-v2f-eps-r}, we get:
\begin{align*}
\zeta^0(x) = \frac{\mathfrak{D}}{\sigma^*}\frac{{\rm d}\rho}{{\rm d}x}
\end{align*}
Upon passing to the limit (as $\eps\to0$) in \eqref{eq:1d:equation-for-zeta-eps}, we have the following equation:
\begin{align*}
-\frac{{\rm d}\zeta^0}{{\rm d}x} = g(x) - \rho(x).
\end{align*}
Substituting for $\zeta^0$ in the above equation, we get:
\begin{align*}
\rho(x) - \frac{{\rm d}}{{\rm d}x}\left( \frac{\mathfrak{D}}{\sigma^*} \frac{{\rm d}\rho}{{\rm d}x}\right) = g(x)
\qquad
\mbox{ for }x\in (-\ell, +\ell).
\end{align*}
The unique solvability of the limit equation \eqref{eq:1d:diff-limit-equation}-\eqref{eq:1d:diff-limit-BC} follows from the Lax-Milgram theorem.
\end{proof}

\section{Arbitrary dimensions}\label{sec:arbitrary}

In the previous section, using the moments method, we managed to prove the $\eps\to0$ limit in the one-dimensional setting under the assumption that the exponent $\beta\le2$. In this section, we prove that the $\eps\to0$ limit for the linear Boltzmann equation \eqref{eq:lbe:model:eps}-\eqref{eq:lbe:absorption:eps} in arbitrary dimensions can be obtained under some smallness criterion on the $H^{-\frac{1}{2}}$ norm of a quotient involving the heterogeneous coefficient $\sigma^\eps(x)$. 

\begin{Thm}\label{thm:H-1/2}
Suppose there exists $\bar{\sigma}(x)\in L^\infty(\Omega)$, bounded away from zero, such that
\begin{align}\label{H-1/2-condition}
\left\| \frac{\bar{\sigma}(x) - \sigma^\eps(x)}{\bar{\sigma}(x)} \right\|_{H^{-1/2}(\Omega)} = \mathcal{O}(\eps^{1+}).
\end{align}
Then the family of local densities $\langle f^\eps(x,\cdot)\rangle$ exhibit the following compactness property:
\[
\int_{\mathcal{V}} f^\eps(x,v)\, {\rm d}\mu(v) \rightharpoonup \rho(x) \qquad \mbox{ weakly in }L^2(\Omega)
\]
where the limit local density $\rho(x)$ satisfies the following diffusion equation:
\begin{align}
\rho(x) - \nabla \cdot \left( \frac{\mathfrak{D}}{\bar{\sigma}(x)} \nabla \rho(x) \right) = g(x) \quad \mbox{ for }x\in \Omega,\label{limit-pb}
\\
\rho(x) = 0 \quad \mbox{ for }x\in \partial\Omega.\label{limit-BC}
\end{align}
with $\mathfrak{D}$ a constant matrix equal to $\langle v\otimes v\rangle$.
\end{Thm}

\begin{proof}
Let us rewrite our steady state model problem \eqref{eq:lbe:model:eps} as follows:
\begin{align*}
\eps\, f^\eps + v\cdot \nabla f^\eps + \frac{\bar{\sigma}(x)}{\eps} \left( f^\eps - \langle f^\eps \rangle\right) = \Big(\bar{\sigma}(x) - \sigma^\eps(x)\Big)\left( \frac{f^\eps - \langle f^\eps \rangle}{\eps}\right) + \eps\, g(x).
\end{align*}
Applying the moments approach outlined in Section \ref{sec:moments} to the above problem, we arrive at the following weak formulation with a smooth test function $\Psi(x)$ vanishing on $\partial \Omega$:
\begin{equation}\label{eq:general:D:wf}
\begin{aligned}
\eps\int_\Omega\frac{1}{\bar{\sigma}(x)} \langle v f^\eps\rangle\cdot \nabla\Psi(x)\, {\rm d}x
+ \int_\Omega \left( \frac{1}{\bar{\sigma}(x)}\nabla \cdot \langle (v\otimes v)f^\eps\rangle \right)\cdot \nabla\Psi(x)\, {\rm d}x
+ \int_\Omega \langle f^\eps\rangle\Psi(x)\, {\rm d}x\\
= \int_\Omega \left(\frac{\bar{\sigma}(x) - \sigma^\eps(x)}{\bar{\sigma}(x)} \right)\frac{\langle vf^\eps\rangle}{\eps}\, {\rm d}x
+ \int_\Omega g(x)\Psi(x){\rm d}x.
\end{aligned}
\end{equation}
We have for the first term on the left hand side of the above equality:
\begin{align*}
\left| \eps\int_\Omega\frac{1}{\bar{\sigma}(x)} \langle v f^\eps\rangle\cdot \nabla\Psi(x)\, {\rm d}x \right|
\le \eps \left\|\bar{\sigma}^{-1}\right\|_{L^\infty(\Omega)} \|\langle vf^\eps\rangle\|_{[L^2(\Omega)]^d} \|\nabla \Psi\|_{[L^2(\Omega)]^d}
\le C \eps^2,
\end{align*}
thanks to Lemma \ref{lem:crucial}. Next, for the first term on the right hand side of \eqref{eq:general:D:wf} we have
\begin{align*}
\int_\Omega \left(\frac{\bar{\sigma}(x) - \sigma^\eps(x)}{\bar{\sigma}(x)} \right)\frac{\langle vf^\eps\rangle}{\eps}\, {\rm d}x 
= \frac{1}{\eps}\llangle[\bigg] \left(\frac{\bar{\sigma} - \sigma^\eps}{\bar{\sigma}} \right)\mbox{Id} , \langle vf^\eps\rangle \rrangle[\bigg]_{[H^{-1/2}(\Omega)]^d, [H^{1/2}(\Omega)]^d}
\end{align*}
which implies:
\begin{align*}
\left| \int_\Omega \left(\frac{\bar{\sigma}(x) - \sigma^\eps(x)}{\bar{\sigma}(x)} \right)\frac{\langle vf^\eps\rangle}{\eps}\, {\rm d}x  \right|
\le \frac{1}{\eps} \left\|\left(\frac{\bar{\sigma} - \sigma^\eps}{\bar{\sigma}} \right)\mbox{Id}  \right\|_{[H^{-1/2}(\Omega)]^d} \left\| \langle vf^\eps\rangle \right\|_{[H^{1/2}(\Omega)]^d}.
\end{align*}
The hypothesis \eqref{H-1/2-condition} and the velocity-averaging Lemma \cite{golse1985resultat, golse1988regularity} together imply that the above term is of $\mathcal{O}(\eps^+)$. With all the above observation, we are left to pass to the limit in the following expression:
\begin{align*}
\mathcal{O}(\eps^2)
+ \int_\Omega \left( \frac{1}{\bar{\sigma}(x)}\nabla \cdot \langle (v\otimes v)f^\eps\rangle \right)\cdot \nabla\Psi(x)\, {\rm d}x
+ \int_\Omega \langle f^\eps\rangle\Psi(x)\, {\rm d}x
= \mathcal{O}(\eps^+)
+ \int_\Omega g(x)\Psi(x){\rm d}x.
\end{align*}
Observe that
\begin{align*}
\lim_{\eps\to0}\int_\Omega \left( \frac{1}{\bar{\sigma}(x)}\nabla \cdot \langle (v\otimes v)f^\eps\rangle \right)\cdot \nabla\Psi(x)\, {\rm d}x
&= 
-\lim_{\eps\to0}
\sum_{i,j=1}^d \int_\Omega \langle v_iv_jf^\eps\rangle \partial_{x_j} \left( \frac{1}{\bar{\sigma}(x)} \partial_{x_i}\Psi(x)\right) \, {\rm d}x\\
&= - \sum_{i,j=1}^d  \int_\Omega \mathfrak{D}_{ij}\rho(x)\partial_{x_j} \left( \frac{1}{\bar{\sigma}(x)} \partial_{x_i}\Psi(x)\right) \, {\rm d}x,
\end{align*}
where $\mathfrak{D}_{ij} = \langle v_i v_j \rangle$. Thus, in the limit, we arrive at the following expression:
\begin{align*}
\int_\Omega \frac{\mathfrak{D}}{\bar{\sigma}(x)}\nabla\rho(x)\cdot \nabla \Psi(x)\, {\rm d}x
+ \int_\Omega \rho(x)\Psi(x)\, {\rm d}x
= \int_\Omega g(x)\Psi(x)\, {\rm d}x, 
\end{align*}
which is nothing but the weak formulation of the limit problem \eqref{limit-pb}-\eqref{limit-BC}. The unique solvability of \eqref{limit-pb}-\eqref{limit-BC} follows again by the application of Lax-Milgram theorem. Hence the entire family converges to the limit local density.
\end{proof}
Next, we make an interesting observation on the smallness of the $H^{-\frac12}$-condition \eqref{H-1/2-condition} with regard to a rapidly oscillating periodic function. Remark that the smallness assumption \eqref{H-1/2-condition} of the $H^{-\frac12}$-norm in the periodic setting corresponds to having the exponent $\beta>2$.
\begin{Lem}\label{lem:H-12:periodic}
Take $\sigma^\eps(x) = 2 + \sin\left(\frac{x}{\eps^\beta}\right)$ with $\beta>2$ and take $\bar{\sigma} = 2$. Then
\begin{align*}
\frac{1}{\eps}
\left\|
\frac{\bar{\sigma}
-
\sigma^\eps }{\bar{\sigma}}
\right\|_{H^{-\frac12}(0,2\pi)}
=
\mathcal{O}(\eps^+)
\end{align*}
\end{Lem}

\begin{proof}
As the denominator in the quotient is a constant, we shall ignore it for the calculations to follow. We shall compute the $H^{-1}$-norm by testing against a function $\varphi\in H^1$:
\begin{align*}
\int_0^{2\pi}
\sin\left(\frac{x}{\eps^\beta}\right)
\varphi(x)
\, {\rm d} x
& =
-
\eps^\beta
\int_0^{2\pi}
\partial_x
\left(
\cos\left(\frac{x}{\eps^\beta}\right)
\right)
\varphi(x)
\, {\rm d} x
\\
& =
\eps^\beta
\int_0^{2\pi}
\cos\left(\frac{x}{\eps^\beta}\right)
\partial_x\varphi(x)
\, {\rm d} x
-
\eps^\beta
\left(
\cos(\eps^{-\beta}2\pi)
\varphi(2\pi)
-
\cos(0)
\varphi(0)
\right)
\end{align*}
which implies that
\begin{align*}
\left\|\sin\left(\frac{x}{\eps^\beta}\right)\right\|_{H^{-1}(0,2\pi)}
\sim \eps^\beta.
\end{align*}
As we have 
\[
\left\|
\sin \left(\frac{x}{\eps^\beta}\right)
\right\|_{L^2(0,2\pi)}
=
\mathcal{O}(1),
\]
interpolating between $H^{-1}$ and $L^2$ implies that
\begin{align*}
\left\|\sin\left(\frac{x}{\eps^\beta}\right)\right\|_{H^{-\frac12}(0,2\pi)}
\sim \eps^\frac{\beta}{2}
\end{align*}
Hence we have
\begin{align*}
\frac{1}{\eps}
\left\|
\bar{\sigma}
-
\sigma^\eps 
\right\|_{H^{-\frac12}(0,2\pi)}
\sim
\eps^{\frac{\beta}{2} - 1}.
\end{align*}
The choice of $\beta>2$ indeed implies that the $H^{-\frac12}$-norm is of $\mathcal{O}(\eps^+)$.
\end{proof}
Even though the result of Lemma \ref{lem:H-12:periodic} is given in one dimension, the proof carries over to any arbitrary dimension.
\begin{Rem}
Note that the smallness condition on the $H^{-1/2}$-norm in Theorem \ref{thm:H-1/2} is quite strong as suggested by Lemma \ref{lem:H-12:periodic} which essentially says that any smooth function depending on the argument $\frac{x}{\eps^\beta}$ would satisfy the smallness assumption \eqref{H-1/2-condition} provided $\beta>2$. Do note that we have treated the case $\beta\le2$ for the one-dimensional case in Theorem \ref{thm:1d}. If we were to suppose that the heterogeneous coefficient has the following structure
\[
\sigma^\eps(x) = \bar{\sigma}(x) + \alpha(\eps) h^\eps(x),
\]
i.e. the highly heterogeneous oscillations are of small amplitude, of size $\alpha(\eps)$, which vanish in the $\eps\to0$ limit then the limit procedure can be carried out without reverting to the condition \eqref{H-1/2-condition}. However, this is trivial as the heterogeneities are dying out in the $\eps\to0$ limit. The smallness assumption \eqref{H-1/2-condition} in Theorem \ref{thm:H-1/2} and the velocity averaging lemma (see Proof of Theorem \ref{thm:H-1/2}) does allow us to have genuine heterogeneity in the coefficients $\sigma^\eps(x)$.
\end{Rem}

\begin{Rem}\label{Rem:Non-stationary setting}
Even though the results presented so far -- Theorem \ref{thm:1d} and Theorem \ref{thm:H-1/2} -- concern the stationary transport model, analogous results for the associated non-stationary model follow straightaway. Consider
\begin{align*}
\partial_t f^\eps = \mathcal{T}^\eps f^\eps,
\qquad
f^\eps(0)=f^{in}(x,v)
\end{align*}
with the operator
\[
\mathcal{T}^\eps f^\eps := -\frac{1}{\eps} v\cdot \nabla_x f^\eps + \frac{\sigma^\eps(x)}{\eps^2} \left( \int_{\mathcal{V}} f^\eps(t,x,w) \, {\rm d}\mu(w) - f^\eps(t,x,v) \right).
\]
The following holds for the semigroup associated with $\mathcal{T}^\eps$:
\[
\left( p - \mathcal{T}^\eps\right)^{-1} = \int_0^\infty e^{-pt} e^{t\mathcal{T}^\eps}\, {\rm d}t.
\]
Inverting the Laplace transform, we get
\[
e^{t\mathcal{T}^\eps} f^{in} = \frac{1}{2\pi i} \lim_{\ell\to\infty} \int_{\gamma-i\ell}^{\gamma+i\ell} e^{pt} \left( p - \mathcal{T}^\eps\right)^{-1} f^{in} \, {\rm d}p.
\]
Taking the $\eps\to0$ limit in the previous expression, we get
\begin{align*}
\lim_{\eps\to0} e^{t\mathcal{T}^\eps} f^{in} 
& = \frac{1}{2\pi i} \lim_{\ell\to\infty} \int_{\gamma-i\ell}^{\gamma+i\ell} e^{pt} \lim_{\eps\to0} \left( p - \mathcal{T}^\eps\right)^{-1} f^{in} \, {\rm d}p.
\end{align*}
The asymptotic limit obtained in Theorem \ref{thm:H-1/2} for the resolvent helps us get
\begin{align*}
\lim_{\eps\to0} e^{t\mathcal{T}^\eps} f^{in} 
& = \frac{1}{2\pi i} \lim_{\ell\to\infty} \int_{\gamma-i\ell}^{\gamma+i\ell} e^{pt} \left( p - \mathcal{D}\right)^{-1} \rho^{in} \, {\rm d}p = e^{t\mathcal{D}} \rho^{in}.
\end{align*}
where we have used the following notations:
\[
\mathcal{D} u := \nabla_x \cdot \left( \frac{\mathfrak{D}}{\bar{\sigma}(x)} \nabla_x u \right)
\qquad
\mbox{ and }
\quad
\rho^{in}(x):= \int_{\mathcal{V}} f^{in}(x,v)\, {\rm d}\mu(v).
\]
\end{Rem}

\section{Concluding remarks}\label{sec:conclude}

We have seen in Theorem \ref{thm:1d} that we can get an explicit expression for the effective diffusion coefficient. This is analogous to the H-limits in one dimensional setting in the theory of H-convergence \cite{murat1997h}. The theory of H-convergence, however, goes beyond the one-dimensional setting in getting explicit expressions while dealing with laminated materials. Our computations show that this is indeed the case in our setting. Results in this flavor will be in a later publication of the authors \cite{bardos2016div-curl}. The main handicap of our result in the $\beta\le2$ case is that we are unable to handle dimensions higher than one. As noted in Remark \ref{rem:H-div}, the estimates are in $H({\rm div};\Omega)$ for the matrix-valued function $\langle(v\otimes v)f^\eps\rangle$. This is in stark contrast to the classical estimates used in the H-convergence theory with regard to elliptic problems. In a work which is in progress \cite{bardos2016div-curl}, the authors have made some progress in getting around the available less regularity information via constructing suitable class of test functions which emphasizes the importance of transport behavior at the scale of the microstructure. This approach employs the famous div-curl lemma (see \cite{dumas2000homogenization} for a kinetic analogue of the div-curl lemma). Finally, it would be interesting to address this simultaneous limit in the case of linear Fokker-Planck equation. The authors shall address this problem in the near future.


\begin{thebibliography}{10}

\bibitem{allaire1992homogenization}
G.~Allaire.
\newblock Homogenization and two-scale convergence.
\newblock {\em SIAM Journal on Mathematical Analysis}, 23(6):1482--1518, 1992.

\bibitem{allaire2002shape}
G.~Allaire.
\newblock {\em Shape optimization by the homogenization method}, volume 146.
\newblock Springer Science \& Business Media, 2002.

\bibitem{allaire1999homogenization}
G.~Allaire and G.~Bal.
\newblock Homogenization of the criticality spectral equation in neutron
  transport.
\newblock {\em ESAIM: Mod{\'e}lisation Math{\'e}matique et Analyse
  Num{\'e}rique}, 33(4):721--746, 1999.

\bibitem{allaire2013transport}
G.~Allaire and F.~Golse.
\newblock Transport et diffusion.
\newblock {\em Lecture notes, Ecole polytechnique}, 2013.

\bibitem{bal2012corrector}
G.~Bal, N.~Ben~Abdallah, and M.~Puel.
\newblock A corrector theory for diffusion-homogenization limits of linear
  transport equations.
\newblock {\em SIAM Journal on Mathematical Analysis}, 44(6):3848--3873, 2012.

\bibitem{bardos2013diffusion}
C.~Bardos, E.~Bernard, F.~Golse, and R.~Sentis.
\newblock The diffusion approximation for the linear boltzmann equation with
  vanishing scattering coefficient.
\newblock {\em Commun. Math. Sci.}, 13(3):641--671, 2015.

\bibitem{bardos1991fluid}
C.~Bardos, F.~Golse, and D.~Levermore.
\newblock Fluid dynamic limits of kinetic equations. i. formal derivations.
\newblock {\em Journal of Statistical Physics}, 63(1-2):323--344, 1991.

\bibitem{bardos2016div-curl}
C.~Bardos and H.~Hutridurga.
\newblock Div-curl approach for diffusion-homogenization limits for linear
  transport equations.
\newblock In preparation, 2016.

\bibitem{abdallah2012diffusion}
N.~Ben~Abdallah, M.~Puel, and M-S. Vogelius.
\newblock Diffusion and homogenization limits with separate scales.
\newblock {\em Multiscale Modeling \& Simulation}, 10(4):1148--1179, 2012.

\bibitem{ben2005diffusion}
N.~Ben~Abdallah and M-L. Tayeb.
\newblock Diffusion approximation and homogenization of the semiconductor
  boltzmann equation.
\newblock {\em Multiscale Modeling \& Simulation}, 4(3):896--914, 2005.

\bibitem{bensoussan1978asymptotic}
A.~Bensoussan, J-L. Lions, and G.~Papanicolaou.
\newblock {\em Asymptotic analysis for periodic structures}.
\newblock North-Holland, Amsterdam, 1978.

\bibitem{dumas2000homogenization}
L.~Dumas and F.~Golse.
\newblock Homogenization of transport equations.
\newblock {\em SIAM Journal on Applied Mathematics}, 60(4):1447--1470, 2000.

\bibitem{golse1991particle}
F.~Golse.
\newblock Particle transport in nonhomogeneous media.
\newblock In {\em Mathematical aspects of fluid and plasma dynamics}, pages
  152--170. Springer, 1991.

\bibitem{golse1988regularity}
F.~Golse, P-L. Lions, B.~Perthame, and R.~Sentis.
\newblock Regularity of the moments of the solution of a transport equation.
\newblock {\em Journal of functional analysis}, 76(1):110--125, 1988.

\bibitem{golse1985resultat}
F.~Golse, B.~Perthame, and R.~Sentis.
\newblock Un r{\'e}sultat de compacit{\'e} pour les {\'e}quations de transport
  et application au calcul de la limite de la valeur propre principale d?un
  op{\'e}rateur de transport.
\newblock {\em CR Acad. Sci. Paris S{\'e}r. I Math}, 301(7):341--344, 1985.

\bibitem{goudon2003homogenization}
T.~Goudon and A.~Mellet.
\newblock Homogenization and diffusion asymptotics of the linear boltzmann
  equation.
\newblock {\em ESAIM: Control, Optimisation and Calculus of Variations},
  9:371--398, 2003.

\bibitem{goudon2001approximation}
T.~Goudon and F.~Poupaud.
\newblock Approximation by homogenization and diffusion of kinetic equations.
\newblock 2001.

\bibitem{murat1978compacite}
F.~Murat.
\newblock Compacit{\'e} par compensation.
\newblock {\em Annali della Scuola Normale Superiore di Pisa-Classe di
  Scienze}, 5(3):489--507, 1978.

\bibitem{murattartar1978}
F.~Murat and L.~Tartar.
\newblock H-convergence.
\newblock {\em S{\'e}minaire d'Analyse Fonctionnelle et Num{\'e}rique de
  l'Universit{\'e} d'Alger, mimeographed notes}, 1978.

\bibitem{murat1997h}
F.~Murat and L.~Tartar.
\newblock H-convergence, topics in the mathematical modelling of composite
  materials, 21--43.
\newblock {\em Progr. Nonlinear Differential Equations Appl}, 31, 1997.

\bibitem{nguetseng1989general}
G.~Nguetseng.
\newblock A general convergence result for a functional related to the theory
  of homogenization.
\newblock {\em SIAM Journal on Mathematical Analysis}, 20(3):608--623, 1989.

\bibitem{sentis1980approximation}
R.~Sentis.
\newblock Approximation and homogenization of a transport process.
\newblock {\em SIAM Journal on Applied Mathematics}, 39(1):134--141, 1980.

\end{thebibliography}

\end{document}